\documentclass[12pt]{article}

\usepackage{color}

\usepackage{amssymb,amsthm,amsmath}
\usepackage{latexsym}
\usepackage{graphicx}
\usepackage{epstopdf}

\usepackage{url}
\usepackage{fullpage}

\definecolor{darkblue}{rgb}{0,0,0.7}
\definecolor{darkgreen}{rgb}{0,0.3,0}
\definecolor{darkred}{rgb}{0.7,0,0}

\newcommand\rd{\color{darkred}}

\newtheorem{theorem}{Theorem}
\newtheorem{lemma}[theorem]{Lemma}

\newcommand\Th[1]{Theorem~\ref{t:#1}}

\newcommand\ex{\ensuremath{\mathrm{ex}}}

\title{Rainbow Tur\'an problems for paths and forests of stars}

\date{}

\author{
Daniel Johnston\thanks{Department of Mathematical Sciences, University of Montana, Missoula, Montana 59801, USA.}
\and
Cory Palmer\thanks{Department of Mathematical Sciences, University of Montana, Missoula, Montana 59801, USA. Research partially supported by University Research Grant Program, University of Montana.}
\and
Amites Sarkar\thanks{Department of Mathematics, Western Washington University, Bellingham, Washington 98225, USA.}
}

\begin{document}

\maketitle

\begin{abstract}
For a fixed graph $F$, we would like to determine the maximum number of edges in a properly edge-colored graph
on $n$ vertices which does not contain a {\it rainbow copy} of $F$, that is, a copy of $F$ all of whose edges
receive a different color. This maximum, denoted by $\ex^*(n,F)$, is the {\it rainbow Tur\'an number} of $F$,
and its systematic study was initiated by Keevash, Mubayi, Sudakov and Verstra\"ete in 2007~\cite{KMSV}. We
determine $\ex^*(n,F)$ exactly when $F$ is a forest of stars, and give bounds on $\ex^*(n,F)$ when $F$
is a path with $k$ edges, disproving a conjecture in~\cite{KMSV}.
\end{abstract}

\section{Introduction}

For a fixed graph $F$, we would like to determine the maximum number of edges in a properly edge-colored graph
on $n$ vertices which does not contain a {\it rainbow copy} of $F$, that is, a copy of $F$ all of whose edges
receive a different color. This maximum, denoted by $\ex^*(n,F)$, is the {\it rainbow Tur\'an number} of $F$,
and its systematic study was initiated by Keevash, Mubayi, Sudakov and Verstra\"ete in 2007~\cite{KMSV}.
Among other things they proved that when $F$ has chromatic number at least $3$, then
\[
\ex^*(n,F) = (1+o(1))\ex(n,F)
\]
where $\ex(n,F)$ is the (usual) Tur\'an number of $F$. They also showed that
\[
\ex^*(n,K_{s,t}) = O(n^{2-1/s})
\]
where $K_{s,t}$ is the complete bipartite graph with classes of size $s$ and $t$.
This research was continued by Das, Lee and Sudakov~\cite{DLS}, who partially answered a question from~\cite{KMSV}
on even cycles (this case has an interesting connection to additive number theory). In this paper, we determine
$\ex^*(n,F)$ exactly when $F$ is a forest of stars, and give bounds on $\ex^*(n,F)$ when $F$ is a path
with $l$ edges, disproving a conjecture in~\cite{KMSV}.

Our methods also yield short proofs of the classic results on Erd\H os and Gallai on the (usual)
Tur\'an numbers of matchings~\cite{EG}, and of some recent results of Lidick\'y, Liu and Palmer \cite{LLP} on the Tur\'an
numbers of forests of stars.

\section{Matchings}

Write $M_k$ for a matching with $k$ edges. The usual Tur\'an number for matchings was determined by Erd\H os and
Gallai~\cite{EG}, who proved the following. Define $G_{n,k}=(V,E)$ to be the graph containing a clique $G_k$ on
vertex set $V_k\subset V$, where $|V|=n,|V_k|=k$, and in which each $v\in V_k$ is joined to every vertex of
$W=V\setminus V_k$. Then
\begin{align*}
\ex(n,M_k)=\max\{e(G_{n,{k-1}}),e(K_{2k-1})\}&=\max\left\{\binom{k-1}{2}+(k-1)(n-k+1),\binom{2k-1}{2}\right\}\\
&=n(k-1)+O(k^2),
\end{align*}
and, for sufficiently large $n$, $G_{n,{k-1}}$ is the unique extremal graph. The second term of the maximum is necessary
since a clique on $2k-1$ vertices also contains no $M_k$, and for small $n$ it has more edges than $G_{n,{k-1}}$.

In other words, for sufficiently large $n$, $\ex(n,M_k)=\binom{k-1}{2}+(k-1)(n-k+1)$. Rather surprisingly, the
same is true for $\ex^*(n,M_k)$. First we establish a weak version of this result. Although both the next two
theorems are special cases of the results in the next section, their proofs will serve as templates for what follows.

\begin{theorem}\label{t:weakmatching}
\[
\ex^*(n,M_k)=n(k-1)+O(k^2).
\]
\end{theorem}
\begin{proof}
Suppose $G=(V,E)$ has the maximum number of edges such that there exists a proper edge-coloring $\chi$ of $G$
with no rainbow $M_k$. Then $G$ must contain a rainbow $M_{k-1}$, on vertex set $A$, say. Write $B=V\setminus A$,
$C\subset A$ for those vertices of $A$ which send at least $t=2k$ edges to $B$, and set $c=|C|$.

We must have $c\le k-1$, or else we could greedily build a rainbow matching from $A$ to $B$ of size $k$ as follows.
First choose an edge $c_1b_1\in E$, where $c_1\in C$ and $b_1\in B$, where without loss of generality $\chi(c_1b_1)=1$.
Then choose an edge $c_2b_2\in E$ of a different color, say $\chi(c_2b_2)=2$, where $c_2\in C$ and $b_2\in B$
with $b_2\not=b_1$. This is possible since $d(c_2)\ge 3$. Continuing, we finally choose $c_kb_k\in E$ with
$\chi(c_kb_k)=k$, which is possible since $d(c_k)\ge 2k-1$ (we have $k-1$ vertices $b_1,\ldots,b_{k-1}$ and $k-1$
edge colors to avoid). Alternatively, the inequality $c\le k-1$ follows on observing that if any edge $c_ic_j$ of our
$M_{k-1}$ has two vertices from $C$, then $c_ic_j$ can be replaced by two edges $c_ib_i$ and $c_jb_j$ of new colors.

At least (and in fact, exactly) $k-1-c$ of the edges of our $M_{k-1}$ contain no vertex of $C$; write $M'$
for this set of edges. We claim that $G'=G[B]$ is $(k-1-c)$-colorable. Indeed, it is $(k-1-c)$-colored by $\chi$.
For if $e\in E(G')$ has a color not appearing among the colors of $M'$, we can form a rainbow copy of $M_k$ by
starting with $M'$ and $e$, and then greedily extending from the vertices of $C$ as above (at the last stage we have
$k-1$ colors and at most $(c-1)+2\le(k-2)+2=k$ vertices to avoid). Consequently, the maximum degree in $G[B]$ is at
most $k-1-c$, and so $e(G[B])\le\frac{k-1-c}{2}(n-2(k-1))$. Therefore,
\begin{align*}
e(G)&\le\binom{2(k-1)}{2}+(2(k-1)-c)(2k-1)+c(n-2(k-1))+\frac{k-1-c}{2}(n-2(k-1))\\
&=(k-1)(6k-5)-c(2k-1)+\frac{k-1+c}{2}(n-2(k-1))\\
&\le (k-1)(6k-5)+(k-1)(n-2(k-1))\\
&= n(k-1)+(k-1)(4k-3).
\end{align*}
\end{proof}

\noindent Next we refine this argument to get an exact result, at least for sufficiently large $n$.

\begin{theorem}\label{strongmatching}
For $n\ge 9k^2$,
\[
\ex^*(n,M_k)=\binom{k-1}{2}+(k-1)(n-k+1).
\]
\end{theorem}
\begin{proof}
We already know that $\ex^*(n,M_k)\ge\ex(n,M_k)=\binom{k-1}{2}+(k-1)(n-k+1)$, so we only need to show that
$\ex^*(n,M_k)\le\binom{k-1}{2}+(k-1)(n-k+1)$. To this end, suppose again that $G=(V,E)$ has the maximum number of
edges such that there exists a proper edge-coloring $\chi$ of $G$ with no rainbow $M_k$. Following the proof of
\Th{weakmatching}, we see that we must have $c=k-1$, since otherwise
\[
e(G)\le\frac{2k-3}{2}(n-2(k-1))+(k-1)(6k-5)<\binom{k-1}{2}+(k-1)(n-k+1),
\]
as long as $n\ge 9k^2$. Armed with this information, we deduce that $G[(A\cup B)\setminus C]$
contains no edges. Otherwise, if $e\in E(G[(A\cup B)\setminus C])$, we could greedily extend $e$ to a rainbow
matching $M_k$ using the vertices of $C$. Consequently,
\[
e(G)\le\binom{|C|}{2}+|C|(|A|-|C|+|B|)=\binom{k-1}{2}+(k-1)(n-k+1).
\]
\end{proof}

\noindent We remark that this method can be used to prove Erd\H os and Gallai's result that
$\ex(n,M_k)=\binom{k-1}{2}+(k-1)(n-k+1)$, at least for sufficiently large $n$. Rather than elaborate here, we note that
the theorem is a special case of the result of Lidick\'y, Liu and Palmer on star forests, which we will reprove
in the next section. Note also that our argument avoids Hall's theorem.

\section{Forests of stars}

In this section we address the rainbow Tur\'an number of a forest $F$ where each component is a star.
In this case, the Tur\'an number was determined by Lidick\'y, Liu and Palmer~\cite{LLP}. We give a new proof of this
result at the end of this section.

Let $F$ be a forest of $k$ stars $S_1,S_2,\dots, S_k$ such that $e(S_j) \leq e(S_{j+1})$ for each $j$. We will construct
a family of $n$-vertex graphs that each have a proper edge-coloring with no rainbow copy of $F$. For $0\leq c \leq k-1$,
define $f(c)$ to be
\[
f(c)=\left(\sum_{i=1}^{k-c} e(S_i)\right)-1.
\]
The graph $H_F(n,c)$ is defined as follows. For $c=k-1$, we connect a set $C$ of $c=k-1$ universal vertices to an edge-maximal
graph $H$ of maximum degree $f(c)=f(k-1)= e(S_1)-1$ on the remaining $n-k+1$ vertices. (A universal vertex is one that is
joined to every other vertex, so that in particular $G[C]$ is a clique.)
When $c\leq k-2$, we connect a set $C$ of $c$ universal vertices to an edge-maximal $f(c)$-edge-colorable graph $H$ on $n-c$
vertices.

Note the slight distinction in the definition of the subgraph $H$ in the two cases $c=k-1$ and $c\leq k-2$. In both cases,
it is easy to see that $H$ can only contain $k-c-1$ of the stars in $F$. The remaining $c+1$ stars must each use at least one
vertex from $C$, which is impossible. Therefore, in both cases, $H_F(n,c)$ does not contain a rainbow copy of $F$.

When $c=k-1$, the subgraph $H$ is $(e(S_1)-1)$-regular when either $n-c$ or $e(S_1)-1$ is even. Otherwise, $H$ has
one vertex of degree $e(S_1)-2$ and $n-k$ vertices of degree $e(S_1)-1$. Therefore, the total number of edges in $H_F(n,k-1)$
is
\begin{align*}
e(H_F(n,k-1)) & = \binom{k-1}{2}+(k-1)(n-k+1)+ \left\lfloor \frac{(e(S_1)-1)(n-k+1)}{2} \right\rfloor.
\end{align*}
When $c \leq k-2$, there are exactly $\lfloor \frac{n-c}{2} \rfloor$ edges of each color in $H$, so that $H$ has
$f(c)\lfloor \frac{n-c}{2} \rfloor$ edges. Therefore, the total number of edges in $H_F(n,c)$ is
\begin{align*}
e(H_F(n,c)) & = \binom{c}{2}+c(n-c)+ f(c)\left\lfloor \frac{n-c}{2} \right\rfloor \\
& = \binom{c}{2}+c(n-c) + \left(\left(\sum_{i=1}^{k-c} e(S_i)\right)-1\right)\left\lfloor \frac{n-c}{2} \right\rfloor.
\end{align*}
Consequently, for all $c\leq k-1$, the number of edges in the graph $H_F(n,c)$ is
\begin{equation}\label{edge-bound}
e(H_F(n,c)) = cn + \frac{1}{2}\left(\left(\sum_{i=1}^{k-c} e(S_i)\right)-1\right)n+O(1).
\end{equation}
Furthermore, the subgraph $H$ of $H_F(n,c)$ has average degree $f(c)-\epsilon$, where $\epsilon < 1$.

Of particular interest is the construction $H_F(n,0)$, which is simply an edge-maximal $(e(F)-1)$-edge-colored graph,
since $f(0)=e(F)-1$.

The key to our analysis is the following technical lemma, which allows us to restrict our attention to the family $H_F(n,c)$.

\begin{lemma}\label{general-stars}
Let $F$ be a forest of $k$ stars. Suppose that $G$ is an edge-maximal properly edge-colored graph on $n$ vertices containing
no rainbow copy of $F$. Then, for sufficiently large $n$, $G$ is isomorphic to one of the graphs $H_F(n,c)$.
\end{lemma}

Before turning to the proof of this lemma, we explain its use in the proof of our main result, Theorem~\ref{main-star-theorem}.
Specifically, suppose we have proved Lemma~\ref{general-stars}, and consider a fixed forest of stars $F$. In order to find the
extremal graphs for a rainbow copy of $F$, we just need to determine the value of $c=c(F)$ that maximizes the number of edges
$e(H_F(n,c))$ of $H_F(n,c)$.

For example, when $F$ is a forest of stars each of size $1$ (i.e., a matching), then, for large $n$, the sum in (\ref{edge-bound})
is maximized when $c=k-1$. Therefore, for large $n$, an edge-maximal properly edge-colored graph $G$ containing no rainbow copy of
$F$ must be isomorphic to $H_F(n,k-1)$. In this case, $f(k-1)=e(S_1)-1=0$ (this holds whenever $F$ contains a star of size 1), so that
$G$ consists of a universal set of size $k-1$ joined to an independent set of size $n-k+1$. This reproves Theorem~\ref{strongmatching}.

It turns out that, for every $F$, the maximum of $e(H_F(n,c))$ is attained at either $c=0$ or $c=k-1$.

\begin{theorem}\label{main-star-theorem}
Let $F$ be a forest of $k$ stars. Suppose that $G$ is an edge-maximal properly edge-colored graph on $n$ vertices containing
no rainbow copy of $F$. Then, for sufficiently large $n$,
1) if $F$ contains no star of size $1$, then $G$ is isomorphic to $H_F(n,0)$;

\noindent 2) otherwise, $G$ is isomorphic to the larger of $H_F(n,0)$ and $H_F(n,k-1)$.
\end{theorem}

\begin{proof}
First consider the case when $F$ contains no star of size $1$. In this case, if $F$ contains at least one star of size at least $3$,
then, for sufficiently large $n$, the right hand side of (\ref{edge-bound}) is maximized when $c=0$. Therefore, by Lemma~\ref{general-stars},
$G$ must be isomorphic to $H_F(n,0)$ (for large $n$).
		
If every star in $F$ has size $2$, then the sum of the two main terms in (\ref{edge-bound}) is constant over all $c\leq k-1$,
so we need to examine the error term. In both the cases $c=k-1$ and $c\leq k-2$, we have
\[
e(H_F(n,c))=\binom{c}{2}+c(n-c) +\left(2(k-c)-1\right)\left\lfloor \frac{n-c}{2} \right \rfloor.
\]
Simple computations show that this is maximized at $c=0$. Therefore, $G$ must be isomorphic to $H_F(n,0)$.

To summarize, if $F$ contains no star of size 1, $G$ must be isomorphic to $H_F(n,0)$, if $n$ is sufficiently large. As already
mentioned, this extremal graph is just an edge-maximal graph that is properly edge-colored with $f(0)=e(F)-1$ colors.
	
Now suppose that $F$ contains a star of size $1$. Write $s\geq 1$ for the number of stars of size $1$, $t$ for the number of stars
of size $2$, and $p=k-s-t$ for the number of stars of size at least $3$ in $F$. If $p=0$, then we should clearly take $c=k-1$ to
maximize the sum of the two main terms in (\ref{edge-bound}). Consequently, we may assume $p>0$. We now have three estimates for
the number of edges in $H_F(n,c)$, depending on the value of $c$.
If $c<p$ (and $p>0$), then
\[
e(H_F(n,c)) = cn + \frac{1}{2}\left(s+2t+\left(\sum_{i=s+t+1}^{k-c} e(S_i)\right)-1\right)n+O(1),
\]
which is maximized (for large $n$) when $c=0$ (as each $e(S_i)$ in the above sum is at least $3$).
Thus, when $c<p$ (and $p>0$), we should take $c=0$, and then
\begin{equation}\label{small-c}
e(H_F(n,c)) = \frac{1}{2}\left(s+2t+\left(\sum_{i=s+t+1}^{k}e(S_i)\right)-1\right)n+O(1).
\end{equation}
If next $p \leq c < p+t$, then
\begin{equation}\label{med-c}
e(H_F(n,c)) = cn + \frac{1}{2}(s+2(t-(c-p))-1)n+O(1) = \frac{1}{2}(s+2t+2p-1)n+O(1),
\end{equation}
which (for large $n$) is clearly smaller than (\ref{small-c}) if $p>0$.
If lastly $p+t\leq c\leq p+t+s-1 = k-1$, then
\[
e(H_F(n,c)) = cn + \frac{1}{2}(s-(c-(p+t))-1)n+O(1) = \frac{1}{2}(s+t+p+c-1)n+O(1),
\]
which is maximized (for large $n$) when $c=k-1$. (We remind the reader that in the case we are considering, $f(k-1)=e(S_1)-1=0$,
so that both constructions of $H_F(n,c)$ coincide when $c=k-1$.) Thus, when $p+t\leq c \leq p+t+s-1=k-1$, we should take $c=k-1=s+t+p-1$,
and then
\begin{equation*}
e(H_F(n,c)) = (s+t+p-1)n+O(1) = (k-1)n+O(1),
\end{equation*}
which is larger than (\ref{med-c}) when $n$ is large. Therefore, for sufficiently large $n$, the number of edges in $H_F(n,c)$
is maximized when $c$ is either $0$ or $k-1$.
\end{proof}

The choice of $c$ to maximize the sum of the two main terms in (\ref{edge-bound}) can be illustrated as follows (see Table 1).
Write down a row of $k$ 2s, and underneath this row, write down the star sizes $e(S_k),e(S_{k-1}),\ldots,e(S_1)$ in decreasing
order. Next, take the sum of the first $c$ entries in the top row and the last $k-c$ entries in the bottom row, where $c\le k-1$.
This sum represents twice the coefficient of $n$ in (\ref{edge-bound}).

\begin{table}
\[\begin{array}{|cccc|c||ccccc||ccccc|}\hline
&&p&&&&&t&&&&&s&&\\\hline
{\bf\rd 2}&{\bf\rd 2}&{\bf\rd 2}&{\bf\rd 2}&2&2&2&2&2&2&2&2&2&2&2\\
5&4&4&3&{\bf\rd 3}&{\bf\rd 2}&{\bf\rd 2}&{\bf\rd 2}&{\bf\rd 2}&{\bf\rd 2}&{\bf\rd 1}&{\bf\rd 1}&{\bf\rd 1}&{\bf\rd 1}&{\bf\rd 1}\\\hline
\end{array}\]
\caption{Illustration of the proof of Theorem~\ref{main-star-theorem}}
\end{table}

We now turn our attention to the proof of Lemma~\ref{general-stars}. We begin with a simple lemma.

\begin{lemma}\label{degree-lemma}
Fix positive integers $d$ and $\Delta$ and a constant $0\leq\epsilon<1$. If $G$ is a graph with average degree at least $d-\epsilon$
and maximum degree at most $\Delta$, then the number of vertices in $G$ of degree less than $d$ is at most
\[
\frac{\Delta-d+\epsilon}{\Delta-d+1} n.
\]
In particular, the number of vertices in $G$ of degree at least $d$ is $\Omega(n)$ (i.e. at least $Cn$ where  $C=C(d,\Delta,\epsilon)>0$).
\end{lemma}

\begin{proof}
The sum of the degrees in $G$ is at least $(d-\epsilon) n$.
On the other hand, if $x$ is the number of vertices of degree less than $d$ in $G$, then the sum of the degrees in $G$ is at most
\[
(d-1)x + \Delta(n-x).
\]
Combining these two estimates and solving for $x$ gives the result.
\end{proof}

We are now ready to prove Lemma~\ref{general-stars}.

\begin{proof}[Proof of Lemma~\ref{general-stars}]
Let $G$ be as in the statement of the theorem, and let $C$ be the set of vertices in $G$ of degree at least $3e(F)$.
Write $c=|C|$. Observe that $c\leq k-1$, since otherwise we could greedily embed the components of $F$ into $G$,
using the vertices of $C$ as their centers.
	
The subgraph $G'=G[V\setminus C]$ has maximum degree at most $3e(F)$. Since $G$ has at least as many edges as the graph
$H_F(n,c)$, it follows that $G'$ must have average degree at least $f(c)-\epsilon$, for some $\epsilon < 1$. Therefore,
by Lemma~\ref{degree-lemma}, the subgraph $G'$ has at least $\Omega(n)$ vertices of degree
\[
f(c)= \left(\sum_{i=1}^{k-c} e(S_i)\right)-1.
\]
	
Now suppose (for a contradiction) that $G'$ has a vertex $v$ of degree greater than $f(c)$. Then we can form a rainbow
copy of $F$ in $G$ as follows. Choose $k-c-1$ vertices of $G'$ of degree $f(c)$ that are at distance at least $3$ from
each other and from $v$ (this is possible since the maximum degree is constant). We can build a rainbow forest of the
stars $S_1,S_2,\dots, S_{k-c-1}$ on these vertices, since these stars use $f(c)+1-e(S_{k-c})$ edge colors. The vertex
$v$ has degree at least $f(c)+1$, so it is incident to at least $f(c)+1-(f(c)+1-e(S_{k-c})) = e(S_{k-c})$ unused colors.
Therefore, we can extend the rainbow forest to include $S_{k-c}$. Finally, the remaining $c$ stars of $F$ can be greedily
embedded using the vertices in $C$ as their centers, so that $G$ contains a rainbow copy of $F$. This is a contradiction.
Therefore, $G'$ has maximum degree at most $f(c)$. When $c=k-1$ we are done, since we have shown that $G$ has at most
as many edges as $H_F(n,k-1)$.
	
Let us now consider the case $c\leq k-2$. The lower bound $e(G)\ge e(H_F(n,c))$ shows that the number of edges in $G'$
is at least
\[
f(c)\left \lfloor \frac{n-c}{2} \right \rfloor\ge f(c)\left(\frac{n-c}{2} \right)-\left\lfloor\frac{f(c)}{2}\right\rfloor.
\]
In particular, $G'$ has $n-O(1)$ vertices of degree $f(c)$, since $G'$ has maximum degree $f(c)$. We claim that $G'$ must
be colored with $f(c)$ edge colors. Suppose, for a contradiction, that $G'$ is colored with at least $f(c)+1$ colors.
Then there is a color class, say {\em red}, with at most
\[
\frac{1}{f(c)+1}\left \lfloor \frac{n-c}{2} \right \rfloor
\]
edges. Therefore, there are $\Omega(n)$ vertices in $G'$ of degree $f(c)$ that are not incident to a red edge.
	
Since $c \leq k-2$, the sum in $f(c)$ has at least two terms, so that
\[
2e(S_1) \leq e(S_1) + e(S_2) \leq \sum_{i=1}^{k-c} e(S_i)=f(c)+1.
	\]
As $e(S_1)$ is an integer, this implies that $e(S_1) \leq \lceil f(c)/2 \rceil$.
	
We now embed $S_1$ in $G'$ using a red edge. If $n-c$ is even, then every vertex in $G'$ has degree
$f(c) \geq \lceil f(c)/2 \rceil$, so we can choose a vertex $v$ incident to a red edge and embed $S_1$ using that red edge.
	
When $n-c$ is odd, $G'$ may contain vertices of degree less than $f(c)$. Consider a red edge $uv$ and observe that at least
one of the vertices $u$ and $v$ (say $v$) has degree at least $\lceil f(c)/2 \rceil$; otherwise the number of edges in $G'$
is less than $f(c)\left \lfloor \tfrac{n-c}{2} \right \rfloor$. Therefore, we can embed $S_1$ using the red edge $uv$ with
$v$ as the center.

Now, among the vertices not incident to red edges, pick $k-c-1$ vertices of degree $f(c)$ that are at distance at least $3$
from each other and from the center $v$ of $S_1$. Using these vertices as centers, we can greedily build a rainbow forest of
stars $S_2,S_3,\dots, S_{k-c}$, since we have only used at most $e(S_1)-1$ of the $f(c)$ colors incident to these vertices.
Finally, the remaining $c$ stars of $F$ can be greedily embedded using the vertices in $C$ as their centers, so that $G$
contains a rainbow copy of $F$. This is a contradiction. Therefore, $G'$ is properly $f(c)$-edge-colored.
\end{proof}

We now give a new proof of the result of Lidick\'y, Liu and Palmer on the Tur\'an number of forests of stars.

We begin by describing the extremal graph for the forest of stars $S_1,S_2,\dots, S_k$, where $e(S_j) \leq e(S_{j+1})$ for each $j$.
Let $H'_F(n,i)$ be the graph obtained by connecting a set of $i$ universal vertices to an edge-maximal graph of maximal degree
$e(S_{k-i})-1$ on $n-i$ vertices. Observe that if one of $e(S_{k-i})-1$ or $n-i$ is even, and $n$ is large enough, then $H$ is
$(e(S_{k-i})-1)$-regular. If both are odd, then $H$ has exactly one vertex of degree $e(S_{k-i})-2$, and $n-i-1$ vertices of degree
$e(S_{k-i})-1$. Each of the graphs $H'_F(n,i)$ is $F$-free, since otherwise each of the $i+1$ stars $S_k,S_{k-1},\dots, S_{k-i}$ must use at
least one vertex from the universal set of size $i$, which is impossible.

\begin{theorem}[Lidick\'y, Liu, Palmer \cite{LLP}]
Let $F$ be a forest of $k$ stars $S_1,S_2,\dots, S_k$, such that $e(S_{j}) \leq e(S_{j+1})$ for each $j$. Then
\[
\ex(n,F) = \max_{0 \leq i \leq k-1} \left \{i(n-i) + \binom{i}{2} + \left\lfloor \frac{(e(S_{k-i})-1)(n-i)}{2}\right\rfloor \right \}.
\]
\end{theorem}

\begin{proof}
Note that $G$ has at least as many edges as $H'_F(n,i)$ for all $i \leq k-1$.
Suppose that $G$ has a set $C$ of $c$ vertices of degree at least $e(F)$. We must have $c \leq k-1$, since otherwise we could greedily
embed $F$ from the vertices of $C$. Let $G'=G[V\setminus C]$ be the graph on the remaining $n-c$ vertices.
The maximum degree of $G'$ is less than $e(F)$. First let us suppose that $c=k-1$. In this case, we claim that the maximum degree of $G'$
is at most $e(S_1)-1$. Indeed, if there is a vertex $v$ of higher degree, then we can embed $S_1$ into $G'$ using $v$, and complete
the forest $F$ by greedily embedding the stars $S_2,S_3,\dots S_k$ using the vertices of $C$ as their centers.
	
Next suppose that $c<k-1$. Suppose (for a contradiction) that $e(S_{k-c-1})= e(S_{k-c})$. Comparing $G$ to $H'_F(n,c+1)$, we see that $G'$ must
have average degree at least $e(S_{k-c-1})-\epsilon=e(S_{k-c})-\epsilon $. Therefore, by Lemma~\ref{degree-lemma}, the graph $G'$ contains
$\Omega(n)$ vertices of degree at least $e(S_{k-c})$. Now we can embed $F$ as follows. Choose $k-c$ vertices of $G'$ of degree $e(S_{k-c})$ that
are at distance at least $3$ from each other. We can embed the stars $S_1,S_2,\dots, S_{k-c}$ on these vertices. Next we can greedily embed
the remaining stars $S_{k-c+1}, \dots, S_k$ into $G$ using the vertices of $C$ as their centers; a contradiction.
	
Therefore, we may assume that $e(S_{k-c-1})<e(S_{k-c})$. By comparing $G$ to $H'_F(n,c)$, we see that $G'$ must have average degree at least
$e(S_{k-c})-1$. Therefore, by Lemma~\ref{degree-lemma}, the graph $G'$ contains $\Omega(n)$ vertices of degree at least $e(S_{k-c})-1$. Now suppose
that $G'$ has a vertex $v$ of degree greater than $e(S_{k-c})-1$. Then we can embed $F$ as follows. Choose $k-c-1$ vertices of $G'$ of degree
$e(S_{k-c})-1$ that are at distance at least $3$ from each other and from $v$. We can embed the stars $S_1,S_2,\dots, S_{k-c-1}$ on these
vertices, since $e(S_{k-c})-1 \geq e(S_{k-c-1})$. Next we embed the star $S_{k-c}$ at $v$, and then greedily embed the remaining stars
$S_{k-c+1}, \dots, S_k$ into $G$ using the vertices of $C$ as their centers; a contradiction. Therefore, the maximum degree of $G'$ is
$e(S_{k-c})-1$.
\end{proof}

\section{Paths}

In this paper, $P_l$ will denote a path with $l$ {\it edges}, which we will call a path of length $l$.
The usual Tur\'an number for paths was determined asymptotically by Erd\H os and Gallai~\cite{EG}, and exactly
by Faudree and Schelp~\cite{FS}. Erd\H os and Gallai proved that, given a path length $l$, if $l$ divides $n$
then
\[
\ex(n,P_l)=\frac{n}{l}\binom{l}{2}=\frac{l-1}{2}n,
\]
and the unique extremal graph is the disjoint union of $\frac{n}{l}$ copies of $K_l$. We briefly recall the proof.
First we show that any graph $G$ with minimum degree at least $\delta$ contains a path of length
$2\delta$ (provided of course that $2\delta<n$). Next, consider a graph $G$ of order $n$ with more than $\frac{l-1}{2}n$
edges (i.e., of average degree greater than $l-1$). By repeatedly removing a vertex of minimum degree, we can show that
$G$ must contain a subgraph $H$ whose minimum degree is at least $\frac{l}{2}$, and so $H$ contains a path of length $l$.

Following this approach for the rainbow Tur\'an problem therefore requires us to find a {\it rainbow} path of length
$c\delta$ in a graph of minimum degree $\delta$. To this end, we have the following theorem, which generalizes a result
of Gy\'arf\'as and Mhalla~\cite{GM}, and is itself a special case of a theorem of Babu, Chandran and Rajendraprasad~\cite{BCR}.
For completeness, we provide a short proof of the result we need, which is less technical than the proof in~\cite{BCR}.

\begin{theorem}\label{t:2/3}
Let $G$ be a graph with minimum degree $\delta=\delta(G)$. Then any proper edge-coloring of $G$ contains a rainbow
path of length at least $\frac23\delta$.
\end{theorem}
\begin{proof}
Suppose that $c$ is a proper edge-coloring of $G$.
Take a longest rainbow path $P=v_0v_1\cdots v_l$ in $G$, of length $l$. Without loss of generality, $c(v_{i-1}v_i)=i$
for each $i$ (i.e., the $i^{\rm th}$ edge of $P$ receives color $i$). Write $s_o$ for the number of edges colored with
colors $1,\ldots,l$ that $v_0$ sends to vertices outside $P$, and note that $v_0$ can send no other edges outside $P$,
or else $P$ could be extended. Also write $s_i$ for the number of edges of colors $1,\ldots,l$ that $v_0$ sends to other
vertices of $P$ (including $v_1$), and write $s^{\times}$ for the number of edges of other colors that $v_0$ sends to
vertices of $P$. Finally, define $t_o,t_i$ and $t^{\times}$ to be the analogous quantities for $v_l$.

Observe now that
\[
s_o+s_i\le l,\eqno(1)
\]
since $c$ is a proper coloring, that
\[
s_i+s^{\times}\le l,\eqno(2)
\]
since there are exactly $l$ vertices on $P$ other than $v_0$, and that
\[
s_o+t^{\times}\le l,\eqno(3)
\]
since if $v_iv_l\in E(G)$ with $c(v_iv_l)>l$ then there is no $w\not\in V(P)$ with $c(wv_0)=c(v_iv_{i+1})=i+1$, or else
$wv_0v_1\cdots v_iv_lv_{l-1}\cdots v_{i+1}$ would be a rainbow path in $G$ of length $l+1$. Analogous inequalities hold
for $t_o,t_i$ and $t^{\times}$.

Consequently, combining (1), (2) and (3) with the minimum degree condition, we have
\[
2\delta\le (s_o+s_i+s^{\times})+(t_o+t_i+t^{\times})=(s_i+s^{\times})+(s_o+t^{\times})+(t_o+t_i)\le l+l+l=3l,
\]
so that $l\ge\frac23\delta$, as desired.
\end{proof}

We remark that the constant $\frac23$ cannot be improved in general. To see this, let $G$ be the disjoint union of $r$ copies of $K_4$,
and properly 3-color the edges of each $K_4$ (there is a unique way to do this, up to isomorphism). Then $\delta(G)=3$,
and the longest rainbow path in $G$ has length 2. However, Chen and Li~\cite{CL}, and independently Mousset~\cite{M},
proved that a proper edge-coloring of $K_n$ contains a rainbow path of length $\frac34n-o(n)$. It is widely believed
(see~\cite{A})
that a proper edge-coloring of $K_n$ in fact contains both a rainbow path and a rainbow cycle of length $n-o(n)$, and perhaps
even a rainbow path of length $n-2$. However, Maamoun and Meyniel~\cite{MM} showed that we are not always guaranteed a
rainbow path of length $n-1$. In their construction, $n=2^k$, and we identify the vertices of $K_{2^k}$ with the points
of the Boolean cube $\{0,1\}^k$. If we now color each edge ${\bf uv}$ with color ${\bf u-v}\not={\bf 0}$, a monochromatic
path ${\bf v_0v_1\cdots v_{n-1}}$ of length $n-1$ in $K_n$ would involve all possible colors (except for ${\bf 0}$), so that
\[
{\bf v_0-v_{n-1}}=\sum_{i=0}^{n-2}({\bf v_i-v_{i+1}})=\sum_{{\bf 0\not=x}\in\{0,1\}^k}{\bf x}=\sum_{{\bf x}\in\{0,1\}^k}{\bf x}={\bf 0},
\]
which implies that $v_0=v_{n-1}$, a contradiction.

A slight modification of the proof of \Th{2/3} yields a short proof of the full result of Babu, Chandran and Rajendraprasad~\cite{BCR}
mentioned above. Their result deals with general (not necessarily proper) edge-colorings, in which, given an edge-colored graph $G$,
$\theta(G)$ is the minimum number of distinct colors seen at each vertex. Clearly $\theta(G)=\delta(G)$ if the coloring is proper.

\begin{theorem}\label{t:2/3gen}
Let $G$ be an edge-colored graph in which every vertex is incident to at least $\theta=\theta(G)$ edge-colors. Then $G$ contains a rainbow
path of length at least $\frac23\theta$.
\end{theorem}
\begin{proof}
We follow the proof of \Th{2/3}, with a slight change in the definitions of $s_o,s_i$ and $s^{\times}$. This time, $s_o$ is the number
of {\it colors} of edges that $v_0$ sends to vertices outside $P$ (as before, each of these colors already occurs on $P$), and $s^{\times}$ is
the number of colors not seen on $P$ which occur as the colors of edges $v_0$ sends to $P$. Now $s_i$ is the number of colors
from 1 to $l$ that occur as colors of edges $v_0$ sends to $P$ and {\it which are not counted in} $s_o$. The rest of the proof goes
through as before, with $\delta$ replaced by $\theta$.
\end{proof}

Returning to the problem at hand, we can use \Th{2/3} to obtain a bound on the rainbow Tur\'an number of paths.

\begin{theorem}\label{t:paths}
For each fixed $l\ge 1$, we have
\[
\frac{l-1}{2}n\sim\ex(n,P_l)\le\ex^*(n,P_l)\le\left\lceil\frac{3l-2}{2}\right\rceil n.
\]
\end{theorem}
\begin{proof}
We will make use of the standard fact that a graph $G$ of average degree more than $2d$ contains a subgraph $H$ of minimum degree at least $d+1$.
This is proved by repeatedly removing a vertex of minimum degree from $G$.

First, suppose that $l$ is even, and write $l=2k$. Let $G$ be a graph of order $n$ with more than $\frac{3l-2}{2}n=(3k-1)n$ edges (and so of
average degree more than $2(3k-1)$). Then $G$ contains a subgraph $H$ of minimum degree at least $3k$, which by \Th{2/3} contains a rainbow
path of length $2k=l$.

Second, suppose that $l$ is odd, and write $l=2k+1$. Let $G$ be a graph of order $n$ with more than $\frac{3l-1}{2}=(3k+1)n$ edges (and so
of average degree more than $2(3k+1)$). Then $G$ contains a subgraph $H$ of minimum degree at least $3k+2$, which by \Th{2/3} contains a rainbow
path of length $2k+1=l$.
\end{proof}

For small values of $l$, one can do considerably better. It is trivial that $\ex^*(n,P_1)=\ex(n,P_1)=0$ and that
$\ex^*(n,P_2)=\ex(n,P_2)=\left\lfloor\frac{n}{2}\right\rfloor$. When $l=3$, we have the following simple result.

\begin{theorem}
Suppose that $n$ is divisible by 4. Then $\ex^*(n,P_3)=\frac{3n}{2}=\frac32\ex(n,P_3)+O(1)$.
\end{theorem}
\begin{proof}
The example already shown, namely $\frac{n}{4}$ disjoint copies of properly 3-colored $K_4$s, shows that $\ex^*(n,P_3)\ge\frac{3n}{2}$.
For the other direction, suppose that $G=(V,E)$ is a graph with more than $\frac{3n}{2}$ edges and no rainbow $P_3$, and select $v\in V$
with $d(v)\ge 3$ (there must be at least one such $v$). Then the neighbors $v_1,\ldots,v_r$ of $v$ can only be adjacent to each other,
since if $v_iw\in E$ with $vw\not\in E$ then $wv_ivv_j$ is a rainbow $P_3$ for some $j$ (chosen so that the colors of $v_iw$ and $vv_j$ are
different). Moreover, if $d(v)\ge 4$, then $G[v\cup\Gamma(v)]$ is a star, since if $v_iv_j\in E$ then $v_jv_ivv_k$ is a rainbow $P_3$,
where this time $k$ has been chosen so that $v_iv_j$ and $vv_k$ receive different colors. Consequently, if $d(v)\ge 3$, then
$G_v=G[v\cup\Gamma(v)]$ is a component of $G$ whose average degree is at most 3, so we may remove it and apply induction.
\end{proof}

\noindent For $P_4$, we have the following theorem.

\begin{theorem}
If $n$ is divisible by 8, then $\ex^*(n,P_4)= 2n$. In general, $\ex^*(n,P_4)=2n+O(1)$.
\end{theorem}
\begin{proof}
The lower bound comes from the proper edge-coloring of $K_{4,4}$ illustrated in Figure 1, which contains no rainbow $P_4$.
(To see this, note that in the given coloring, any 4-cycle containing two identically-colored edges must in fact be 2-colored,
so that every 4-cycle contains either 2 or 4 colors. Now suppose (to the contrary) that $xyzst$ is a rainbow $P_4$. Then the cycle $xyzsx$ must contain
all 4 colors, so that edges $st$ and $sx$ must receive the same color, which is impossible since they are adjacent.) Next, if
$n=8k$, then the disjoint union of $k$ such edge-colored $K_{4,4}$s has $2n$ edges and no rainbow $P_4$.
Consequently, $\ex^*(n,P_4)\ge 2n$ if $8|n$, and $\ex^*(n,P_4)\ge 2n+O(1)$ in general.

\begin{figure}[htp!]
\centering
\includegraphics[scale=1]{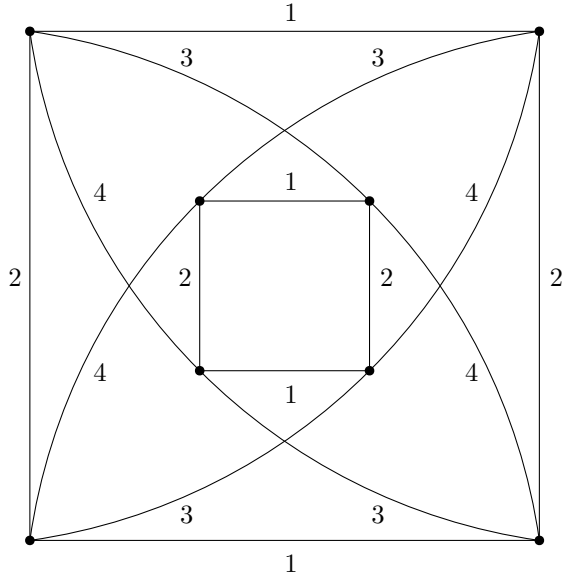}
\caption{\small A proper edge-coloring of $K_{4,4}$ with no rainbow $P_4$}
\label{f:blockingset}
\end{figure}

For the upper bound, we show that every proper edge-coloring of an $n$-vertex graph $G$ with $m > 2n$ edges contains a rainbow $P_4$.

As noted before, $G$ contains a subgraph $G'$ of minimum degree at least $3$, since otherwise we can repeatedly
remove vertices of degrees $1$ and $2$ so that the average degree increases. Furthermore, $G'$ has average degree greater than $4$. Therefore, $G'$ has a vertex $v$ of degree at least $5$. We will show that $G'$ contains a rainbow $P_4$. The proof now splits into two cases.

\noindent{\bf Case 1: $G'$ contains a rainbow $P_3$ ending at $v$.} This case is illustrated in Figure 2; let the rainbow $P_3$
be $P=vxyz$, where edges $vx,xy$ and $yz$ are colored $1$, $2$ and $3$ respectively. Since $v$ has degree at least $5$, it must be adjacent to at
least $2$ vertices not on $P$; suppose these vertices are $s$ and $t$. If either of the edges $vs$ and $vt$ receives a color other than $2$ or $3$, then we have a rainbow $P_4$. Now suppose that $c(vs)=2$ and $c(vt)=3$, where $c$ denotes the color of the edge. If $v$ is adjacent
to any other vertex $u$ not on $P$, then since $c(uv)$ would have to be different from $1$, $2$ and $3$, the edge $uv$ with $P$ forms a rainbow $P_4$. Otherwise, the vertex $v$ has degree $5$ and is adjacent to both $y$ and $z$. Without loss of generality, suppose $c(vy)=4$ and
$c(vz)=5$.

Suppose that the vertex $z$ is adjacent to $x$. Note
that $c(xz)$ cannot be $1$, $2$ or $3$, and so $svxzy$ is a rainbow $P_4$.
If $z$ is not adjacent to $x$, then $z$ is adjacent to a vertex $w$ not on $P$ (possibly $w=s$ or $w=t$) as the minimum degree of $G'$ is at least $3$.
We know that
$c(wz)$ cannot be $3$ or $5$; if $c(wz)=1$ then $wzvyx$ is a rainbow $P_4$, while if $c(wz)=2$ then $wzyvx$ is a rainbow $P_4$.
However, if $c(wz)$ is not 1, 2 or 3, then $vxyzw$ is a rainbow $P_4$. 
Accordingly, this completes the proof in Case 1.

\begin{figure}[htp!]
\centering
\includegraphics[scale=1]{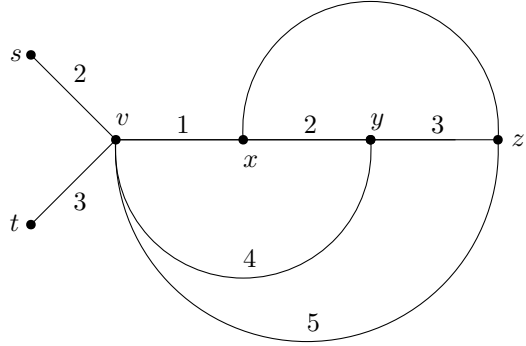}
\caption{\small A rainbow $P_3$ ending at a vertex $v$ of degree at least $5$}
\label{f:blockingset}
\end{figure}

\noindent{\bf Case 2: $G'$ contains no rainbow $P_3$ ending at $v$.} Since $\delta(G')\ge 3$, $G'$ contains a rainbow $P_2$
ending at $v$; let this path be $vxy$, where $c(vx)=1$ and $c(xy)=2$. The vertex $y$ has degree at least 3; if $y$ were adjacent to
two vertices $s$ and $t$ other than $v$ and $x$, then one of edges $ys$ and $yt$ would receive color 3, creating a rainbow
$P_3$ ending at $v$. Consequently, the degree of $y$ is $3$ and $y$ is adjacent to $v$ and a new vertex $z$. Furthermore, $c(yz)=1$, and, without
loss of generality, $c(yv)=3$. Let $P$ be the path $vxyz$.

The vertex $z$ is adjacent to at most one vertex $w$ not on $P$ and the edge $zw$ must receive color $3$ to avoid
the rainbow $P_3$ $vyzw$ ending at $v$. Consequently, $z$ is adjacent to at least one of $v$ or $x$. The proof now splits into
three sub-cases.

\noindent{\bf Case 2A: $z$ is adjacent to $x$ and a new vertex $w$.} This case is illustrated on the left of Figure 3.
Edge $xz$ cannot receive any of colors $1$, $2$ or $3$, and so $vxzw$ is a rainbow $P_3$ ending at $v$.

\noindent{\bf Case 2B: $z$ is adjacent to $v$ and a new vertex $w$.} This case is illustrated in the center of Figure 3.
Edge $vz$ must receive color 2 to avoid the rainbow $P_3$ $vzyx$ ending at $v$. Now, if $w$ were adjacent to two vertices
$s$ and $t$ other than $v,x,y$ and $z$, then one of edges $ws$ and $wt$ would receive color other than $2$ and $3$, creating a rainbow
$P_3$ ending at $v$. Therefore, there is at least one edge from $w$ to $v$, $x$, or $y$. Such an edge cannot receive colors $1$, $2$, or $3$. If $wv$ is an edge, then $vwzy$ is a rainbow $P_3$; if $wx$ is an edge,
then $vxwz$ is a rainbow $P_3$; if $wy$ is an edge, then $vxyw$ is a rainbow $P_3$. In all cases we have found a rainbow
$P_3$ ending at $v$.

\noindent{\bf Case 2C: $z$ is adjacent to both $v$ and $x$.} This case is illustrated on the right of Figure 3. In this case,
the vertices $v,x,y,z$ induce a properly $3$-edge-colored $K_4$ as otherwise we can easily find a rainbow $P_3$ ending at $v$. We will exploit the resulting symmetry in the three colors
$1$, $2$ and $3$. The vertex $v$ must be adjacent to a new vertex $u$, and, without loss of generality, $c(uv)=4$. If the vertex $u$ is adjacent to a new vertex $w$, then
we may assume that $c(uw)=1$, and then $wuvzx$ would be a rainbow $P_4$. Otherwise,
$u$ is adjacent to at least two of $x,y$ and $z$; suppose it is adjacent to $x$. Then $c(ux)$ cannot be $1$, $2$, $3$ or $4$, and then
$xuvzy$ is a rainbow $P_4$.

Thus, in all three sub-cases we obtain either a rainbow $P_3$ ending at $v$ (leading us to Case 1), or a rainbow $P_4$
in $G'$.

\begin{figure}[htp!]
\centering
\includegraphics[scale=1]{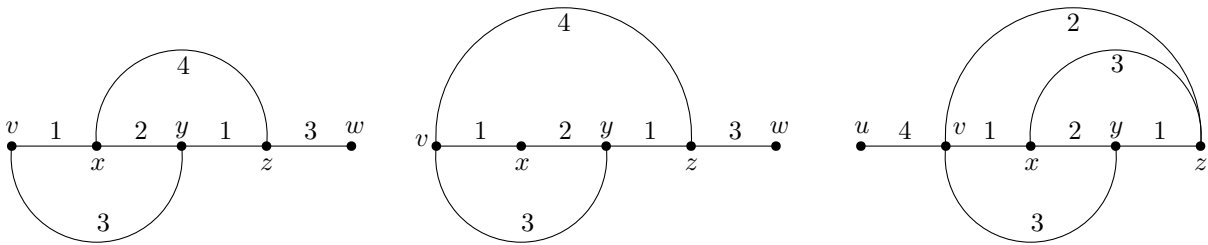}
\caption{\small No rainbow $P_3$ ends at a vertex $v$ of degree at least $5$}
\end{figure}

\end{proof}

Keevash, Mubayi, Sudakov and Verstra\"ete conjectured that the extremal example for rainbow $P_l$s is a disjoint union of
cliques of size $c(l)$, where $c(l)$ is chosen as large as possible so that $K_{c(l)}$ can be properly edge-colored with no
rainbow $P_l$. It is not hard to show that a properly edge-colored $K_5$ must contain a rainbow $P_4$, so that $c(4)=4$.
Consequently, the conjecture implies that $\ex^*(n,P_4)=\tfrac{3n}{2}+O(1)$, which is false, as our theorem shows.

\end{document}